\documentclass[11pt]{article}
\usepackage[top=0.8in, bottom=1in, left=1.2in, right=1.2in]{geometry}

\usepackage[utf8]{inputenc}
\usepackage{amsthm,amsmath}
\usepackage{amssymb, mathrsfs, color}
\usepackage{enumerate}
\usepackage{verbatim}
\usepackage{tikz-cd}
\usepackage{enumitem}
\usepackage{indentfirst}
\usepackage{hyperref}
\hypersetup{
    colorlinks=true,
    linkcolor=blue,
    citecolor=brown,
    filecolor=magenta,
    urlcolor=blue,
}
\usepackage{titlesec}
\usepackage{fancyhdr}
\usepackage{mathtools}
\usepackage{CJKutf8}
\usepackage{enumitem}

\usetikzlibrary{decorations.pathmorphing}

\newtheorem{thm}{Theorem}

\newtheorem{lem}[thm]{Lemma}
\newtheorem*{lem*}{Lemma}
\newtheorem{prop}[thm]{Proposition}

\newtheorem{innercustomgeneric}{\customgenericname}
\providecommand{\customgenericname}{}
\newcommand{\newcustomtheorem}[2]{%
  \newenvironment{#1}[1]
  {%
   \renewcommand\customgenericname{#2}%
   \renewcommand\theinnercustomgeneric{##1}%
   \innercustomgeneric
  }
  {\endinnercustomgeneric}
}

\newcustomtheorem{thm*}{Theorem}
\newcustomtheorem{prop*}{Proposition}
\newcustomtheorem{cor*}{Corollary}

\newtheorem*{thm**}{Theorem}

\theoremstyle{definition}
\newcustomtheorem{defn*}{Definition}

\newtheorem{remark}[thm]{Remark}

\newenvironment{remark*}[2][Remark]{\begin{trivlist}
\item[\hskip \labelsep {\bfseries #1}\hskip \labelsep {\bfseries #2.}]}{\end{trivlist}}

\newcommand{\B}{\mathbb{B}}
\newcommand{\C}{\mathbb{C}}
\newcommand{\Co}{\mathscr{C}}

\newcommand{\R}{\mathbb{R}}

\newcommand{\V}{\mathcal{V}}

\newcommand{\Z}{\mathbb{Z}}

\newcommand{\eps}{\varepsilon}

\newcommand{\Span}{\operatorname{Span}}

\newcommand{\loc}{\mathrm{loc}}
\newcommand{\supp}{\operatorname{supp}}
\newcommand{\Coorvec}[1]{\frac\partial{\partial#1}}

\begin{document}
\author{Yao, Liding\\\small University of Wisconsin-Madison}
\title{A Counterexample to $C^k$-regularity for the Newlander-Nirenberg Theorem}
\date{}
\maketitle
\begin{abstract}
We give an example of $C^k$-integrable almost complex structure that does not admit a corresponding $C^{k+1}$-complex coordinate system.
\end{abstract}

The celebrated Newlander-Nirenberg theorem \cite{NN} states that given an integrable almost complex structure, it is locally induced by some complex coordinate system.

Malgrange \cite{Malgrange} proved the existence of such complex coordinate chart when the almost complex structure is not smooth, and he obtained the following sharp H\"older regularity for this chart:
\begin{thm}[Sharp Newlander-Nirenberg]\label{SharpNN}
Let $k\in\Z_+$ and $0<\alpha<1$, let $M^{2n}$ be a $C^{k+1,\alpha}$-manifold endowed with a $C^{k,\alpha}$-almost complex structure $J:TM\to TM$. If $J$ is integrable, then for any $p\in M$ there is a $C^{k+1,\alpha}$-complex coordinate chart $(w^1,\dots,w^n):U\to\C^n$ near $p$ such that $J\Coorvec{w^j}=i\Coorvec{w^j}$ for $j=1,\dots,n$.
\end{thm}
There are several equivalent characterizations for integrability. One of which is the vanishing of the \textbf{Nijenhuis tensor} $N_J(X,Y):=[JX, JY ]-J[JX,Y ]-J[X,JY ]-[X,Y ]$.

Our main theorem is to show that this is not true when $\alpha\in\{0,1\}$:
\begin{thm}\label{THM}
Let $k,n\in\Z_+$. There is a $C^k$-integrable almost complex structure $J$ on $\R^{2n}$, such that there is no $C^{k,1}$-complex coordinate chart $w:U\subset\R^{2n}\to\C^n$  near $0$ satisfying $J\Coorvec{w^j}=i\Coorvec{w^j}$ for $j=1,\dots,n$.
\end{thm}
For convenience we use the viewpoint of eigenbundle of $J$: Set $\V_n=\coprod_p\{v\in \C T_p\R^{2n}:J_pv=iv\}$. So $J\Coorvec{w^j}=i\Coorvec{w^j}$ for all $j$ iff $d\bar w^1,\dots,d\bar w^n$ spans $\V_n^\bot|_U\le\C T^*\R^{2n}|_U$. And $J$ is integrable if and only if $X,Y\in\Gamma(\V_n)\Rightarrow[X,Y]\in\Gamma(\V_n)$ for all complex vector fields $X,Y$. See \cite{Involutive} Chapter 1 for details.

\medskip
First we can restrict our focus to the 1-dimensional case:
\begin{proof}[Proof of 1-dim $\Rightarrow$ n-dim]Suppose $J_1$ is a $C^k$-almost complex structure on $\C^1_{z^1}$ (not compatible with the standard complex structure), such that near $0$, there is no $C^{k,1}$-complex coordinate $\varphi$ satisfying $\V_1^\bot=\Span d\bar \varphi$ in the domain. Here $\V_1^\bot$ is the dual eigenbundle of $J_1$. 


\medskip
Denote $\theta=\theta(z^1)$ as a $C^k$ 1-form on $\C^1_{z^1}$ that spans $\V_1^\bot$. 

Consider $\R^{2n}\simeq\C^n_{(z^1,\dots,z^n)}$. We identify $\theta$ as the $C^k$ 1-form on $\R^{2n}$.
Take an $n$-dim almost complex structure on $\R^{2n}$ such that the dual of eigenbundle $\V_n^\bot$ is spanned by $\theta,d\bar z^2,\dots,d\bar z^n$. 

In other words, $\V_n$ is the ``tensor'' of $\V_1$ with the standard complex structure of $\C^{n-1}_{(z^2,\dots,z^n)}$.

If $w=(w^1,\dots,w^n)$ is a corresponding $C^{k,1}$-complex chart for $\V_n$ near $0$, then there is a $1\le j_0\le n$ such that $d\bar w^{j_0}\not\equiv0\pmod{d\bar z^2,\dots,d\bar z^n}$ near $0$. In other words, we have linear combinations $d\bar w^{j_0}=\lambda_1\theta+\lambda_2d\bar z^2+\dots+\lambda_nd\bar z^n$ for some non-vanishing function $\lambda_1(z^1,\dots,z^n)$ near $z=0$.

Therefore $w^{j_0}(\cdot,0^{n-1})$ is a complex coordinate chart defined near $z^1=0\in\R^2$ whose differential spans $\V_n^\bot|_{\R^2_{z^1}}\cong\V_1^\bot$ near $0$. By our assumption on $\V_1$, we have $w^{j_0}(\cdot,0)\notin C^{k,1}$. So $w^{j_0}\notin C^{k,1}$, which means $w\notin C^{k,1}$.
\end{proof}

Now we focus on the one-dimensional case. Note that a 1-dim structure is automatically integrable.

Fix $k\ge1$. Define an almost complex structure by setting its eigenbundle $\V_1\le\C T\R^2$ equals to the span of $\Coorvec z+a(z)\Coorvec{\bar z}$, where $a\in C^k(\R^2;\C)$ has compact support that satisfies the following:

\medskip
\begin{enumerate}[nolistsep,label=(\roman*)]
    \item\label{1} $a\in C^\infty_\loc(\R^2\backslash\{0\};\C)$;
    \item\label{2} $\partial_z^{-1}\partial_{\bar z}a\notin C^{k-1,1}$ near 0;
    \item\label{3} $za\in C^{k+1}(\R^2;\C)$ and $z^{-1}a\in C^{k-1}(\R^2;\C)$ (which implies $a(z)=o(|z|)$ as $z\to0$);
    \item\label{4} $\supp a\subset\B^2$;
    \item\label{5} $\|a\|_{C^0}<\delta_0$ for some small enough $\delta_0>0$ (take $\delta_0=10^{-1}$ will be ok).
\end{enumerate}

\medskip
Here we take $\partial_z^{-1}$ to be the \textbf{conjugated Cauchy-Green operator} on the unit disk\footnote{Throughout the paper, $\B^2=B^2(0,1)$ refers to the unit disk in $\C^1$.}: $$\displaystyle \partial_z^{-1}\phi(z)=\partial_{z,\B^2}^{-1}\phi(z):=\frac1\pi\int_{\B^2}\frac{\phi(\xi+i\eta)d\xi d\eta}{\bar z-\xi+i\eta}.$$

We use notation $\partial_z^{-1}$ because it is an right inverse of $\partial_z$. And $\partial_z^{-1}:C^{m,\beta}(\B^2;\C)\to C^{m+1,\beta}(\B^2;\C)$ is bounded linear for all $m\in\Z_{\ge0}$, $0<\beta<1$. See \cite{Vekua} theorem 1.32 in section 8.1 (page 56), or \cite{Shaw} lemma 2.3.4 for example.

\medskip
We can take  $\supp a\subset\B^2$ such that when $|z|<\frac12$, 
\begin{equation}\label{a}
    \textstyle a(z):=\frac1{100}\bar z^{k+1}\partial_z\big((-\log|z|)^{\frac12}\big)=\frac1{100}\partial_z\big(\bar z^{k+1}(-\log|z|)^{\frac12}\big).
\end{equation}

Note that for this $a$ we have $a(z)=O\big(|z|^k(-\log|z|)^{-\frac12}\big)=o(|z|^k)$.

\begin{remark}
Roughly speaking, Property \ref{1} $\operatorname{Singsupp}a=\{0\}$ says that the regularity of $a(z)$ corresponds to the vanishing  order of $a$ at $0$. To some degree, by multiplying with $a(z)$, a function gains some regularity at the origin.
\end{remark}

 We check Property \ref{2} that $\partial_z^{-1}\partial_{\bar z}a\notin C^{k-1,1}$ here.
\begin{lem}\label{lem}
Let $a(z)$ be given by \eqref{a}, and let $\chi\in C_c^\infty(\frac12\B^2)$ satisfies $\chi\equiv1$ in a neighborhood of 0. Then $\partial_z^{-1}(\chi a_{\bar z})\notin C^{k-1,1}$ near $z=0$.
\end{lem}
\begin{proof}Denote $b(z):=\bar z^{k+1}(-\log|z|)^\frac12$, so $b\in C^\infty(\frac12\B^2\backslash\{0\};\C)$, $\chi a=\frac1{100}\chi b_z$ and $\chi a_{\bar z}=\frac1{100}\chi b_{z\bar z}$. 

First we show that $\partial_z^{-1}(\chi b_{z\bar z})-b_{\bar z}\in C^\infty(\frac12\B^2;\C)$. 
We write $$\partial_z^{-1}(\chi b_{z\bar z})- b_{\bar z}=\partial_z^{-1}\partial_z(\chi b_{\bar z})-\chi b_{\bar z}-(1-\chi)b_{\bar z}-\partial_z^{-1}(\chi_zb_{\bar z}).$$

Since $\partial_z\partial_z^{-1}\partial_z(\chi b_{\bar z})=\partial_z(\chi b_{\bar z})$, we know $\partial_z^{-1}\partial_z(\chi b_{\bar z})-\chi b_{\bar z}$ is anti-holomorphic. By Cauchy integral formula we get $\partial_z^{-1}\partial_z(\chi b_{\bar z})-\chi b_{\bar z}\in C^\infty$.

By assumption $0\notin\supp\chi_z$, $0\notin\supp(1-\chi)$ and $b\in C^\infty(\frac12\B^2\backslash\{0\};\C)$, we know $\chi_zb_{\bar z}\in C_c^\infty$ and $(1-\chi)b_{\bar z}\in C^\infty(\frac12\B^2;\C)$. 

Note that when acting on functions supported in the unit disk, $\partial_z^{-1}=\frac1{\pi\bar z}\ast(\cdot)$ is a convolution operator with kernel $\frac1{\pi\bar z}\in L^1$, so $\partial_z^{-1}(\chi_zb_{\bar z})=\frac1{\pi\bar z}\ast(\chi_zb_{\bar z})\in C^\infty$.

It remains to show $b_{\bar z}\notin C^{k-1,1}$ near $0$. Indeed one has $$\partial_{\bar z}^{k+1}(\bar z^{k+1}(-\log|z|)^\frac12)=(k+1)!(-\log|z|)^{\frac12}+O(1)\quad,\text{ as }z\to0,$$ because by Leibniz rule
\begin{align*}
    &\textstyle\partial_{\bar z}^k\partial_{\bar z}\big(\bar z^{k+1}(-\log|z|)^{\frac12}\big)=\sum_{j=0}^{k+1}{k+1\choose j}\partial_{\bar z}^{k+1-j}(z^{k+1})\cdot\partial_{\bar z}^j(-\log|z|)^{\frac12}\\
    =&\textstyle(k+1)!(-\log|z|)^{\frac12}+\sum_{j=1}^{k+1}O(z^j)O\big(z^{-j}(-\log|z|)^{-\frac12}\big)=(k+1)!(-\log|z|)^{\frac12}+O\big((-\log|z|)^{-\frac12}\big).
\end{align*}
\end{proof}
\medskip Now assume $w:\tilde U\subset\R^2\to\C$ is a 1-dim $C^1$-complex coordinate chart defined near $0$ that represents $\V_1$, then $\Span d\bar w=\Span (d\bar z-adz)|_{\tilde U}=\V_1^\bot|_{\tilde U}$. So $d\bar w=\bar w_{\bar z}d\bar z+\bar w_zdz=\bar w_{\bar z}(d\bar z-adz)$, that is,
$$\displaystyle\frac{\partial w}{\partial\bar z}(z)+\bar a(z)\frac{\partial w}{\partial z}(z)=0,\qquad z\in {\tilde U}.$$

\begin{remark}
It is worth noticing that $\partial_w\neq \partial_z+a\partial_{\bar z}$. Indeed $\partial_w$ is only a scalar multiple of $\partial_z+a\partial_{\bar z}$.
\end{remark}

Note that $w_z(0)\neq0$ because $(d\bar z-adz)|_0=d\bar z|_0\in\Span d\bar w|_0$. So by multiplying $w_{ z}(0)^{-1}$, we can assume $w_{z}(0)=1$ without loss of generality.
Then $f:=\log\partial_zw$ is a well-defined function in a smaller neighborhood $U\subset\tilde U$ of $0$, which solves 
\begin{equation}\label{MainEQN}
    \frac{\partial f}{\partial\bar z}(z)+\overline{a(z)}\frac{\partial f}{\partial z}(z)=-\frac{\partial\bar a}{\partial z}(z)\quad\Big(=-\overline{\frac{\partial a}{\partial\bar z}(z)} \Big),\qquad z\in U.
\end{equation}

Property \ref{5} indicates that the operator $\partial_{\bar z}+\bar a\partial_z$ is a first order elliptic operator. Therefore we can consider a second order divergence form elliptic operator $$L:=\partial_z(\partial_{\bar z}+\bar a\partial_z)$$ whose coefficients are $C^k$ globally and are $C^\infty$ outside the origin.

By the classical Schauder's estimate (see \cite{Taylor} Theorem 4.2, or \cite{GT} Chapter 6 \& 8), we have the following:
\begin{lem}[Schauder's interior estimate]\label{Schauder} Assume $u,\psi\in C^{0,1}(\B^2;\C)$ satisfy $Lu=\psi_z$. Let $U\subset\R^2$ be a neighborhood of $0$. The following hold:
\begin{enumerate}[nolistsep,label=(\alph*)]
    \item\label{S1}If $\psi\in C^\infty_\loc(U\backslash\{0\};\C)$, then $u\in C^\infty_\loc(U\backslash\{0\};\C)$.
    \item\label{S2}If $\psi\in C^{k-1,1}(\R^2;\C)$, then $u\in C^{k,1-\eps}_\loc(\R^2;\C)$ for all $0<\eps<1$.
\end{enumerate}
\end{lem}

Our Theorem \ref{THM} for 1-dim case is done by the following proposition:

\begin{prop}
For any neighborhood $U\subset\R^2_{z}$ of $0$, there is no $f\in C^{k-1,1}(U;\C)$ solving \eqref{MainEQN}.
\end{prop}
\begin{proof}
Suppose there is a neighborhood $U\subset\frac12\B^2$ of the origin, and a solution $f\in C^{k-1,1}(U;\C)$ to \eqref{MainEQN}.

Applying Lemma \ref{Schauder} \ref{S1} on \eqref{MainEQN} with $u=f$ and $\psi=-\bar a_z$, we know that$f\in C^\infty(U\backslash\{0\};\C)$.

Take $\chi\in C_c^\infty(U,[0,1])$ such that $\chi\equiv1$ in a smaller neighborhood of $0$. Denote $$g(z):=\chi(z) f(z),\qquad h(z):=\bar zg(z).$$ So $g,h$ are $C^{k-1,1}$-functions defined in $\R^2$ that are also \textbf{smooth away from 0}, and satisfy the following:
\begin{equation}\label{EQNg}
    g_{\bar z}+\bar ag_z=\chi_{\bar z}f+\chi_z\bar af-\chi\bar a_z,
\end{equation}
\begin{equation}\label{EQNh}
    \qquad h_{\bar z}+\bar ah_z=g+\bar z(\chi_{\bar z}f+\chi_z\bar af)-\chi \bar z\bar a_z.
\end{equation}

By construction $\nabla\chi\equiv0$ holds in a neighborhood of $0$, so $\chi_{\bar z}f+\chi_z\bar af\in C_c^\infty(U;\C)$.

\medskip
Under our assumption that $f\in C^{k-1,1}(U;\C)$, then the key is to show that 
\begin{enumerate}[parsep=-0.3ex,label=(\Roman*)]
    \item $h\in C^{k,1-\eps}_c(\R^2;\C)$, $\forall \eps\in(0,1)$. This implies:
    \item\label{C2} $\bar ag_z\in C^{k-1,1-\eps}_c(\R^2;\C)$, $\forall \eps\in(0,1)$.
\end{enumerate}

\medskip\noindent\textbf{(I)} By assumption $\bar z(\chi_{\bar z}f+\chi_z\bar af)\in C^\infty_c$, $g\in C^{k-1,1}$, and by Property \ref{3}, $\chi\bar z\bar a_z\in C^k(\R^2;\C)$. 

Applying Lemma \ref{Schauder} \ref{S2} to \eqref{EQNh}, with $u=h$ and $\psi=g+\bar z(\chi_{\bar z}f+\chi_z\bar a f)-\chi\bar z\bar a_z\in C^k_c(\R^2;\C)$, we get $h\in C^{k,1-\eps}(\B^2;\C)$, for all $0<\eps<1$.

\medskip\noindent\textbf{(II)} When $k\ge2$,  we know  $z^{-1}a\in C^{k-1}$ for Property \ref{3}. So for any $\eps\in(0,1)$, one has $\bar ag_z\in C^{k-1,1-\eps}$ because $$\nabla_{z,\bar z}(\bar ag_z)=\bar a\cdot(\partial_zg_z,\partial_{\bar z}g_z)+g_z\nabla_{z,\bar z}\bar a=\bar z^{-1}\bar a\cdot(h_{zz},h_{z\bar z}-g_z)+O(C^{k-2,1})\in C^{k-2,1-\eps}.$$

When $k=1$, for any $z_1,z_2\in\R^2\backslash\{0\}$, note that $g_{\bar z}$ is smooth outside the origin, so $\bar ag_z\in C^{0,1-\eps}$:
    \begin{align*}
    &|\bar ag_z(z_1)-\bar ag_z(z_2)|\le|\bar a(z_1)||g_z(z_1)-g_z(z_2)|+|\bar a(z_1)-\bar a(z_2)||g_z(z_2)|
    \\
    \le&|\bar z_1^{-1}\bar a(z_1)||\bar z_1g_z(z_1)-\bar z_2g_z(z_2)|+|\bar z_1^{-1}\bar a(z_1)||\bar z_1-\bar z_2||g_z(z_2)|+|\bar a(z_1)-\bar a(z_2)||g_z(z_2)|
    \\
    \le&\|z^{-1}a\|_{C^0}\|\nabla h\|_{C^{1-\eps}}|z_1-z_2|^{1-\eps}+\|z^{-1}a\|_{C^0}\|g\|_{C^{0,1}}|z_1-z_2|+\|a\|_{C^1}\|g\|_{C^{0,1}}|z_1-z_2|.
\end{align*}

\textit{Here as a remark, $|\bar ag_z(z_1)-\bar ag_z(z_2)|\lesssim_{a,g,h}|z_1-z_2|^{1-\eps}$ still makes sense when $z_1$ or $z_2=0$, though $\bar g_z(0)$ may not be defined. Indeed $\lim\limits_{z\to0}\bar ag_z(z)$ exists because $\bar ag_z$ itself has bounded $C^{0,1-\eps}$-oscillation on $\B^2\backslash\{0\}$, and then the limit defines the value of $\bar ag_z$ at $z=0$.}

So for either case of $k$, we have $\bar ag_z\in C^{k-1,1-\eps}(\R^2;\C)$ for all $\eps\in(0,1)$.

\medskip 
Based on consequence \ref{C2}, for \eqref{EQNg} we have 
\begin{equation}\label{FinalEqn}
    g=(g-\partial_{\bar z}^{-1}g_{\bar z})+\partial_{\bar z}^{-1}(\chi_{\bar z}f+\chi_z\bar af)-\partial_{\bar z}^{-1}(\bar ag_z)-\partial_{\bar z}^{-1}(\chi\bar a_z),\qquad\text{in }\B^2.
\end{equation}

The right hand side of \eqref{FinalEqn} consists of four terms, the first to the third are all $C^k$, while the last one is not $C^{k-1,1}$. We explain these as follows:
\begin{itemize}[parsep=-0.3ex]
    \item Since $\partial_z(g-\partial_{\bar z}^{-1}g_{\bar z})=0$ in $\B^2$, we know $g-\partial_{\bar z}^{-1}g_{\bar z}$ is anti-holomorphic, which is smooth in $\B^2$.
    \item By assumption $\chi_{\bar z}f+\chi_z\bar af\in C^\infty_c(\R^2;\C)$, so $\partial_{\bar z}^{-1}(\chi_{\bar z}f+\chi_z\bar af)\in C^\infty(\B^2;\C)$ as well.
    \item By consequence \ref{C2} $\bar ag_z\in C^{k-1,1-\eps}(\R^2;\C)$, for all $\eps\in(0,1)$, so $\partial_{\bar z}^{-1}(\bar ag_z)\in  C^{k,1-\eps}\subset C^k(\B^2;\C)$.
    \item However by Lemma \ref{lem}, $\partial_{\bar z}^{-1}(\chi\bar a_z)\notin C^{k-1,1}$ near 0.
\end{itemize}

Combining each term to the right hand side of \eqref{FinalEqn}, we know $g\notin C^{k-1,1}$ near 0. Contradiction!
\end{proof}

\begin{remark}
The key to the proof is the non-surjectivity of $\partial_z:C^k\to C^{k-1}$, which we use to  construct a function $a(z)$ such that $a(0)=0$, $\operatorname{Singsupp}a=\{0\}$, and $\partial_z^{-1}\partial_{\bar z}a\notin C^k$. 
\end{remark}
\begin{remark}
For positive integer $k$, Malgrange's sharp estimate of Theorem \ref{SharpNN} still holds for Zygmund spaces $\Co^k=B^k_{\infty\infty}$, that is, given $J\in\Co^k$, there exists a $\Co^{k+1}$-coordinate chart $(w^1,\dots,w^n)$, such that $J\Coorvec{w^j}=i\Coorvec{w^j}$. One can also see \cite{SharpElliptic} for details. A reason why our proof does not give a counterexample for Zygmund spaces, is that there does exist a $f\in \Co^k$ defined in a neighborhood of 0 that solves \eqref{MainEQN}.
\end{remark}

\subsection*{Acknowledgement}
The author would like to express his appreciation to his advisor Prof. Brian T. Street for his help.

\bibliographystyle{plain}
\bibliography{Bibliography}

\center{\small\textit{University of Wisconsin-Madison, Department of Mathematics, 480 Lincoln Dr.,\\ Madison, WI, 53706}\\\textit{lyao26@wisc.edu}}

\center{MSC 2010: 35J46 (Primary), 32Q60 and 35F05 (Secondary)}
\end{document}